\documentclass{amsart}
\usepackage{amsfonts}
\usepackage{graphicx}
\usepackage{amscd}
\usepackage{mathtools}

\newtheorem{theorem}{Theorem}
\theoremstyle{plain}

\newtheorem{corollary}{Corollary}

\newtheorem{lemma}{Lemma}

\newtheorem{proposition}{Proposition}
\newtheorem{remark}{Remark}

\numberwithin{equation}{section}

\def\eps{\varepsilon}
\def\Jac{{\rm Jac}}
\def\dis{{\rm dist}}
\def\diam{{\rm diam}}
\def\be #1{\begin{equation}#1\end{equation}}
\def\bes #1{\begin{equation}\begin{split}#1\end{split}\end{equation}}

\graphicspath{    {converted_graphics/}    {/}}

\begin{document}
\title{ Singular SRB measures for a non 1--1 map of the unit square}

\author[P. G\'ora]{Pawe\l\ G\'ora }
\address[P. G\'ora]{Department of Mathematics and Statistics, Concordia
University, 1455 de Maisonneuve Blvd. West, Montreal, Quebec H3G 1M8, Canada}
\email[P. G\'ora]{pawel.gora@concordia.ca}

\author[A. Boyarsky]{Abraham Boyarsky }
\address[A. Boyarsky]{Department of Mathematics and Statistics, Concordia
University, 1455 de Maisonneuve Blvd. West, Montreal, Quebec H3G 1M8, Canada}
\email[A. Boyarsky]{abraham.boyarsky@concordia.ca}

\author[Z. Li]{Zhenyang Li }
\address[Z. Li]{Department of Mathematics, Honghe University, Mengzi, Yunnan 661100, China}
\email[Z. Li]{zhenyangemail@gmail.com}

\thanks{The research of the authors was supported by NSERC grants. The research of Z. Li is also supported
by NNSF of China (No. 11161020 and No. 11361023)}
\subjclass[2000]{37A05, 37A10, 	37E30, 37D20, 37D30}
\date{\today }
\keywords{SRB measures, conditional measures, piecewise linear two-dimensional maps}

\begin{abstract}  
We consider  a map of the unit square which is not 1--1, such as the memory map studied in \cite{MwM1}
Memory maps are 
 defined as follows: $x_{n+1}=M_{\alpha
}(x_{n-1},x_{n})=\tau (\alpha \cdot x_{n}+(1-\alpha )\cdot x_{n-1}),$ where $\tau$ is a one-dimensional map 
on $I=[0,1]$ and  
$0<\alpha <1$ determines how much memory is being used.    In this paper we
let $\tau $ to be the symmetric tent map. To study the dynamics of $M_\alpha$, we consider the two-dimensional map
$$ G_{\alpha }:[x_{n-1},x_{n}]\mapsto [x_{n},\tau (\alpha \cdot x_{n}+(1-\alpha )\cdot x_{n-1})]\, .$$
The   map $G_\alpha$ for $\alpha\in(0,3/4]$ was studied in \cite{MwM1}. In this paper we prove that
for $\alpha\in(3/4,1)$ the map $G_\alpha$ admits a singular Sinai-Ruelle-Bowen measure. We do this 
by applying Rychlik's results for the Lozi map. However, unlike the Lozi map, the maps $G_\alpha$ are not invertible
which creates   complications that we are able to overcome.
\end{abstract}

\maketitle

{Department of Mathematics and Statistics, Concordia University, 1455 de
Maisonneuve Blvd. West, Montreal, Quebec H3G 1M8, Canada}

and

Department of Mathematics, Honghe University, Mengzi, Yunnan 661100, China

\smallskip

E-mails: {abraham.boyarsky@concordia.ca}, {pawel.gora@concordia.ca}, {%
zhenyangemail@gmail.com}.

\section{ Introduction}

Let $\tau $ be a piecewise, expanding map on $I=[0,1]$. We consider a process
$$x_{n+1}=M_{\alpha }(x_{n})\equiv \tau (\alpha \cdot x_{n}+(1-\alpha
)\cdot x_{n-1})\ ,\ 0<\alpha <1, $$
which we call a map with memory since the next state $x_{n+1}$ depends not only on current state $x_n$ but also
on the past $x_{n-1}$. Note that $M_\alpha$ is a map from $[0,1]^2$ to $[0,1]$ and hence is not a dynamical system.

A natural method to study the long term behaviour of the process $M_\alpha$,
is to study the invariant measures of the two dimensional transformation 
\begin{equation*}
G_{\alpha }:[x_{n-1},x_{n}]\mapsto \lbrack x_{n},M_{\alpha
}(x_{n})]=[x_{n},\tau (\alpha \cdot x_{n}+(1-\alpha )\cdot x_{n-1})]\, .
\end{equation*}

In \cite{MwM1} we studied  $G_\alpha$ with the tent map $\tau(x)=1-2|x-1/2|$, $x\in I$,
for $\alpha \in (0,3/4]$. For $0<\alpha \le 0.46,$ we proved that $G_\alpha$ admits an absolutely continuous invariant measure (acim).
We conjecture that acim exists also for  $\alpha\in [0.46,1/2)$.  
As $\alpha $ approaches $1/2$ from below, that is, as we approach a balance
between the memory state and the present state, the support of the acims
become thinner until at $\alpha =1/2$, all points have period 3 or eventually
possess period 3.  
We proved that for $\alpha=1/2$ all points (except two fixed points) are eventually 
periodic with period 3.
For $\alpha=3/4$ we proved that all points of the line $x+y=4/3$ (except the fixed point)
 are of period 2 and all other points (except $(0,0)$) are attracted to this line. 
For $1/2< \alpha < 3/4$, we prove the
existence of a global attractor: for all starting points in the square $[0,1]^2$ except $(0,0)$,
the orbits are attracted to the fixed point $(2/3,2/3).$ 

In this paper, we continue the study of transformation $G_\alpha$ for $\alpha\in (3/4,1)$ and  prove the existence of a singular
Sinai-Ruelle-Bowen measure $\mu_\alpha$. The invariant measure is singular with respect to Lebesgue measure since for $\alpha\in(3/4,1)$ the determinants of the derivative matrices of $G_\alpha$ are less than one,
hence the support of the invariant measure is of Lebesgue measure 0. The invariant measure has two main properties: for Lebesgue almost every point  $x\in [0,1]^2$ and any continuous function $g:[0,1]^2\to \mathbb R$,
$$\lim_{n\to\infty} \frac 1n \sum_{k=0}^{n-1}g(G_\alpha^k x)=\mu_\alpha(g) ,$$
and the conditional measures induced by $\mu_\alpha$ on segments with expanding directions are one-dimensional absolutely continuous measures.

Our method follows closely the techniques in Rychlik \cite{Ry}  for the Lozi map.  The most important difference between the Lozi map and the $G_\alpha$'s is the fact
that our maps are not invertible. For maps that are invertible or locally invertible, there are results known \cite{Alv,Avi,Ben,Bon,Cow,San,Tas,Tsu1,Tsu2,You}. However, to the best of  our knowledge the existence of a singular SRB measure has, until now, not  been proven for any  non-invertible map.

\section{Abstract Reduction Theorem}\label{Abstract}

Similarly as Rychlik in \cite{Ry}, we will start with abstract considerations. Sections \ref{Abstract} and \ref{ACIM} are taken from \cite{Ry} almost without any changes.
We present them here for completeness, to introduce the notation and the results we need in the following sections.

Let $(X,\Sigma, m)$ be a Lebesgue measure with a $\sigma$-algebra $\Sigma$ and a probability measure $m$. Let 
$T:X\to X$ be a measurable, nonsingular mapping, i.e.,  $T_*m\ll m$. We define the Frobenius-Perron operator induced by $T$ as
$$P_T f = \frac {d(T_*(fm))}{dm}  \ \ ({\rm Radon-Nikodym\ derivative)},$$
for $f\in L^1(X,\Sigma, m)$ and we have $P_Tf\in L^1(X,\Sigma, m)$. Equivalently, we can  define $P_Tf$ as the unique element of $L^1(X,\Sigma, m)$
satisfying
$$\int_X (h\circ T)\cdot f dm= \int_X h\cdot P_T f dm ,$$
for all $h\in L^\infty(X,\Sigma, m)$. This means that operator $P_T$ is the conjugate of the Koopman operator $K_T h=h\circ T$ acting on $L^\infty(X,\Sigma, m)$.

A measurable, countable partition $\beta$ of $X$ is called regular iff for every $A\in{\beta}$, $T(A)$ is $\Sigma$-measurable  and $T_{|A}$ maps $(A,\Sigma_{|A})$ onto
$(T(A),\Sigma_{|T(A)})$ isomorphically. For any regular partition ${\beta}$ we define $g_T:X\to \mathbb R^+$ as follows:
\be{\label{def g_T} g_T(x)= \frac {d(T_*(\chi_Am))}{dm} (Tx) \ ,\ {\rm for}\ x\in A\in {\beta}.}
We can write $g_T=\sum_{A\in{\beta}} K_T(P_T \chi_A)\cdot \chi_A$. The function $g_T$ is determined up to a set of measure 0 and  does not depend on the choice of 
partition ${\beta}$. For piecewise differentiable map $T$ the function $g_T$ is the reciprocal of the Jacobian. Using $g_T$ we can express $P_T$ as follows
\be{\label {FPg} P_T f(x)=\sum_{y\in T^{-1}(x)}g_T(y)\cdot f(y)\ ,\ x\in X.}
Equality (\ref{FPg}) holds $m$ almost everywhere.

$$\begin{CD}
Y @> {\phantom{...} S \phantom{...} }>>Y\\
@VV \pi V             @VV \pi V\\
X @> {\phantom{...} T \phantom{...}}>> X
\end {CD}$$

\bigskip 

\bigskip

Now, we consider the case when $T$ is a factor of another mapping $S:Y\to Y$, where $(Y,\Sigma_Y,\nu)$ is a Lebesgue space. We assume that $S$ is nonsingular.
By $\xi$ we denote the measurable partition of $Y$ which is $S$-invariant, i.e., $S^{-1}\xi\le \xi$. Let $X=Y/\xi$ and let $T=S_\xi$ be the factor map.
We assume that $m=\pi_*(\nu)$ or $m=\nu\circ \pi$.
We denote the natural projection by $\pi: Y\to X$. Let $C(x)$ denote the element $\pi^{-1}(x)\in \xi$. We have $S(C(x))\subset C(Tx)$.
We will find the relation between $P_T$ and $P_S$.

$$\begin{CD}
L^1(Y,\Sigma_Y,\nu) @> {\phantom{...} P_S \phantom{...} }>>L^1(Y,\Sigma_Y,\nu)\\
@VV E_\nu(\cdot|\xi) V             @VV E_\nu(\cdot|\xi) V\\
L^1(X,\Sigma,m) @> {\phantom{...} P_T \phantom{...}}>> L^1(X,\Sigma,m)
\end {CD}$$

\bigskip 

\bigskip

\begin{proposition}\label{conditional} (Rychlik \cite{Ry}, Proposition 1) Let $E_\nu(\cdot|\xi):L^1(Y,\Sigma_Y,\nu)\to L^1(X,\Sigma,m)$ be the operator of conditional expectation with respect to the
$\sigma$-algebra generated by the partition $\xi$, see \cite{Bil}. For any $f\in L^1(Y,\Sigma_Y,\nu)$ we have
$$P_T(E_\nu(f|\xi))=E_\nu(P_Sf|\xi). $$
\end{proposition}
\begin{proof} Let $h\in L^\infty(X,\Sigma,m)$. Then:
\begin{equation}\begin{split} &\int_X h\cdot P_T(E_\nu(f|\xi)) dm = \int_X (h\circ T)E_\nu(f|\xi) dm  \\
&=\int_Y (h\circ T\circ \pi)(E_\nu(f|\xi)\circ \pi) d\nu=\int_Y (h\circ  \pi\circ S)f d\nu \\
&=\int_Y (h\circ  \pi) \cdot P_Sf d\nu=\int_Y (h\circ  \pi)(E_\nu(  P_Sf|\xi)\circ \pi) d\nu \\
&=\int_X h\cdot E_\nu(  P_Sf|\xi) dm . 
\end{split}
\end{equation}
We used two properties of the conditional expectation:

(a) If $g\in L^\infty(Y,\Sigma_Y,\nu)$ is $\xi$-measurable, then $E_\nu(gf|\xi)=g E_\nu(f|\xi)$.

(b) $\int_Y  E_\nu(f|\xi) d\nu = \int_Y f d\nu $.
\end{proof}

We assume that $S$ has a regular partition ${\mathcal P}$ with the property 
\be{\label{reg part} S^{-1}\xi\vee {\mathcal P}=\xi.}
 We will consider 
$\xi=\mathcal P^{-}=\bigvee_{k=0}^\infty S^{-k}\mathcal P$ so this property will holds automatically.

\begin{lemma} (Rychlik \cite{Ry}, Lemma 1) The family $\beta=\{\pi(A)\}_{A\in{\mathcal P}}$ is a $T$-regular partition of $X$.
\end{lemma}

\begin{proof} By assumption (\ref{reg part}) ${\mathcal P}\le \xi$ and thus $\pi^{-1}(\pi A)=A$ for every $A\in{\mathcal P}$.
Hence, $\beta$ is a partition. Also, $\pi^{-1}(T(\pi A))=S(A)$. Then, $T(\pi A)$ is measurable, by the definition of the factor space and regularity of $\mathcal P$.
Then, $T_{|\pi(A)}:\pi(A)\to T(\pi(A))$ is the factor $S_{|A}:A\to S(A)$. Moreover, by (\ref{reg part}) $T_{|\pi(A)}$ is 1-1
 (since $S_{|A}$ is almost everywhere 1-1) and  $(T_{|\pi(A)})^{-1}$ is  the factor of 
$(S_{|A})^{-1}$. So, $T_{|\pi(A)}$ is an isomorphism of $(\pi(A), \Sigma_{|\pi(A)})$ and $(T(\pi(A), \Sigma_{|T(\pi(A)})$.
\end{proof}

Let $\{\nu_C\}_{C\in\xi}$ be the family of conditional measures of $\nu$ with respect to $\xi$. In the following proposition we relate $g_S$, $g_T$ and $\{\nu_C\}_{C\in\xi}$.

\begin{proposition}\label{prop:Ry2} (Rychlik \cite{Ry}, Proposition 2) For almost every $x\in X$ and  $\nu_{C(x)}$-almost every $y\in C(x)$, we have
\be{\label{gTgS}g_T(x)=g_S(y)\frac{d((S_{|A})^{-1}_* \nu_{C(Tx)})}{d\nu_{C(x)} }(y) ,}
where $A$ is the element of $\mathcal P$ which contains $C(x)$. Note, that $C(x)=S^{-1}(C(Tx))\cap A$. In particular, $(S_{|A})^{-1}_* \nu_{C(Tx)}$ is equivalent to $\nu_{C(x)}$
for almost every $x\in X$.
\end{proposition}

\begin{proof} Let $h\in L^\infty(Y,\Sigma_Y,\nu)$. Then, 
\be{\label{firstgTgS} E_\nu(P_Sh|\xi)(x)=\int_{C(x)} \sum_{A\in\mathcal P} (h\cdot g_S)\circ (S_{|A})^{-1} d\nu_{C(x)}=\int h d \sigma_{1,x},}
where $$\sigma_{1,x}=\sum_{z\in T^{-1}(x)} g_S\cdot ((S_{|C(z)})^{-1}_* \nu_{C(x)}).$$
The first inequality in (\ref{firstgTgS}) follows by the definition of $P_Sh$ and the fact that $E_\nu(P_Sh|\xi)$ is almost surely constant on elements of $\xi$.
The second, by the definition of $g_S$.
Also, we have
\be{\label{secondgTgS} P_TE_\nu(h|\xi)(x)=\sum_{z\in T^{-1}(x)} \int_{C(z)} h d \nu_{C(z)} g_T(z)=\int h d \sigma_{2,x},}
where $$\sigma_{2,x}=\sum_{z\in T^{-1}(x)}  g_T(z) \, \nu_{C(z)}.$$
In view of Proposition \ref{conditional} since $h$ is arbitrary the measures $\sigma_{1,x}$ and $\sigma_{2,x}$ are equal for almost every $x$. Since the measures $\nu_{C(x)}$ have disjoint supports and since functions $g_T$ 
and $g_S$ are positive almost everywhere the equality (\ref{gTgS}) is proved.
\end{proof}

We will consider situation when the elements of $\xi$ are endowed with some natural measure. Let $\{\ell_C\}_{C\in \xi}$ be a family of such measures such that for any $C\in\xi$ the measure $\ell_C$ is equivalent to
$\nu_C$ and the Radon-Nikodym derivative $$\rho=\frac{d\nu_C}{d\ell_C},$$
defined on $Y$ is $\Sigma_Y$-measurable. Then, for almost every $x\in X$, $(S_{|C(x)})^{-1}_*\ell_{C(Tx)}$ is also equivalent to $\ell_{C(x)}$. Also, the function
$$\lambda(y)=\frac {d((S_{|C(x)})^{-1}_*\ell_{C(Tx)})}{d\ell_{C(x)}}(y)\ ,\ y\in C(x) ,$$
is $\Sigma_Y$-measurable.

Let us now consider the situation when $S:[0,1]^2\to [0,1]^2$ preserves two families of cones, the cone of stable directions and the cone of unstable directions. Let $\xi=\mathcal P^{-}$ be the $S$-invariant partition
which consists of pieces of lines with stable directions. On each element $C\in \xi$ we have Lebesgue measure $\ell_C$. If $\nu$ is the Lebesgue measure on $[0,1]^2$, then we proved in Proposition \ref{prop:Ry9} that
for almost all $C$, the conditional measure $\nu_C$ is equivalent to $\ell_C$.

By Proposition \ref {prop:Ry2} we have
$$g_T\circ \pi=g_S\frac {\rho\circ S}{\rho} \lambda. $$
For $y,y'$ belonging to the same $C$ it gives $$\frac{\rho(y)}{\rho(y')}=\frac {g_S(y)}{g_S(y')}\frac{\rho(Sy)}{\rho(Sy')}\frac{\lambda(y)}{\lambda(y')},$$
and by induction
\be{\label{rhon}
\frac{\rho(y)}{\rho(y')}=\left(\prod_{k=0}^{n-1}\frac {g_S(S^ky)}{g_S(S^ky')}\frac{\lambda(S^ky)}{\lambda(S^ky')}\right)\frac{\rho(S^ny)}{\rho(S^ny')}.}
Formula (\ref{rhon}) proves the following:
\begin{proposition}\label{prop:Ry3} (Rychlik's Proposition 3) For almost every $x\in X$ and for $\nu_{C(x)}$-almost every $y,y'\in C(x)$, the following conditions are equivalent:

(a) $$\lim _{n\to\infty}\frac{\rho(S^ny)}{\rho(S^ny')}=1 ;$$

(b) $$\frac{\rho(y)}{\rho(y')}=\prod_{k=0}^{\infty}\frac {g_S(S^ky)}{g_S(S^ky')}\frac{\lambda(S^ky)}{\lambda(S^ky')}.$$
\end{proposition}

\begin{remark} If (a) holds almost everywhere, then (b) and the condition $\int \rho d\ell_C=1 $ determine $\rho$ completely.
\end{remark}

\section{Existence of Absolutely Continuous Invariant Measures.}\label{ACIM}

As before, let $T$ be a nonsingular map of a Lebesgue space $(X,\Sigma, m)$. Let $\beta$ be a regular partition of $X$ such that $\beta^{-}=\bigvee_{k=0}^\infty T^{-k}(\beta)$ is a partition into points,
i.e., $\beta$ is a generator for $T$. We will give conditions which prove that $T$ admits an invariant measure absolutely continuous with respect to $m$. We introduce the following notations:
$g=g_T$, $g_n=g_{T^n}$, $P=P_T$, for any $A\in \Sigma$, $\beta(A)=\{B\in\beta:m(B\cap A)>0\}$. By supremum and infimum we understand the essential supremum or minimum.

\textbf{(I)\ \phantom{II} Distortion condition:} $$\exists _{\ d>0}\ \forall _{\ n\ge 1} \ \forall _{\ B\in \beta^n} \ \sup_B g_n\le d\cdot \inf_B g_n ;$$

\textbf{(II)\ \phantom{I} Localization condition:} $$\exists _{\eps>0}\ \exists _{\ 0<r<1} \ \forall _{\ n\ge 1} \ \forall _{\ B\in\beta^n}\ m(T^nB)<\eps \Longrightarrow \sum_{B'\in \beta(T^n B)} \sup_{B'} g \le r ;$$

\textbf{(III)\  Bounded variation  condition:} $$\sum_{B\in\beta} \sup_B g <+\infty .$$

\begin{remark} If conditions (I) and (III) hold for $T$, $\beta$ and $g$, then they also hold for $T^N$, $\beta^N$ and $g_N$. Condition (I) holds with the same value of $d$.
Moreover, $$\sum_{B\in\beta^N} \sup_B g_N\le \left(\sum_{B\in\beta} \sup_B g\right)^N .$$
\end{remark}

\begin{theorem}\label{TH1} (Rychlik \cite{Ry}, Theorem 1) Let (I)--(III) be satisfied. Then, the sequence $(P^n \bold 1)_{n\ge 1}$ is bounded in $L^\infty(X,\Sigma, m)$ and the averages
$\frac 1n \sum_{k=0}^{n-1} P^k \bold 1$ converge in $L^1(X,\Sigma, m)$ to some $\phi\in L^\infty(X,\Sigma, m)$ such that $P\phi=\phi$.
\end{theorem}

\begin{proof} For the proof we refer to  \cite{Ry} or to \cite{BG}.
\end{proof}

Theorem \ref{TH1} gives the existence of an absolutely continuous invariant measure. To improve on this result we introduce the following condition:

\textbf{(IV)\  Expanding  condition:} $$\exists _{\ r\in (0,1)}\ \sup_X g \le r .$$

We can  assume that $r$ is chosen to satisfy both (II) and (IV).

\begin{theorem}\label{TH2} (Rychlik \cite{Ry}, Theorem 2) Let (I)--(IV) be satisfied. Then, there exists a bounded, finite dimensional projection $Q: L^1(X,\Sigma, m)\to L^\infty(X,\Sigma, m)$ such that

(1) \ $Q(L^1(X,\Sigma, m))\subset L^\infty(X,\Sigma, m)$ and $Q$ is bounded as an operator from $L^1(X,\Sigma, m)$ to $L^\infty(X,\Sigma, m)$;

(2) for every $f\in L^1(X,\Sigma, m)$ the averages
$\frac 1n \sum_{k=0}^{n-1} P^k f$ converge in $L^1(X,\Sigma, m)$ to $Qf$.

(3) The range  $\mathcal R(Q)$ of $Q$ consists of all eigenvectors of $P$ corresponding to the eigenvalue 1 and of 0 function.

(4) There exist non-negative functions $\phi_1,\phi_2,\dots,\phi_s\in \mathcal R(Q)$ which span $\mathcal R(Q)$ and $\phi_i\wedge\phi_j=0$ as $i\not= j$. Moreover, $\int_X\phi_i dm=1$, $i=1,\dots,s$ and if
$$C_i=\bigcup_{n=0}^\infty T^{-n}\{x: \phi_i(x)>0\} $$
is the basin of attraction of the measure $\phi_i m$, then $Q$ can be represented as
$$Qf=\sum_{i=1}^s \int_{C_i} f \, dm \cdot \phi_i . $$
Moreover, $\bigcup_{i=1}^s C_i=X$ up to a set of measure 0 and $\{\phi_i\}_{i=1}^s$ are the only functions $\phi\in L^1(X,\Sigma, m)$ such that $\phi\cdot m$ is a $T$-invariant, ergodic, probabilistic measure.
\end{theorem}

\begin{proof} For the proof we  refer to  \cite{Ry}.
\end{proof}

\section{ Preliminary Results for Maps $G_\alpha$ when $\alpha\in(\frac{3}{4}, 1)$.}\label{sec:Prelim}

Recall the map $G=G_\alpha$:
\begin{equation*}
G(x,y)=[y,\tau (\alpha y+(1-\alpha ) x)]\, \ , \ (x,y)\in [0,1]^2 ,
\end{equation*}
where $\tau$ is the tent map $x\mapsto 1-2|x-1/2|$.  $L$ denotes the line $\alpha y+(1-\alpha ) x=1/2$
which divides $[0,1]^2$ into two domains on which $G$ is 1--1. A more explicit formula for $G$ is
\begin{equation*}
G(x,y)=\begin{cases}(y, 2 (\alpha y+(1-\alpha ) x)&, \text{ if } y \text{ is below } L ;\\
(y, 2-2 (\alpha y+(1-\alpha ) x)&, \text{ if } y \text{ is above } L .
\end{cases}
\end{equation*}
 We have two possibilities for the  Jacobian matrices:
\[A_{\pm}=DG=\left[ \begin {array}{cc} 0&1\\ \noalign{\medskip}\left( 1-\alpha \right) \tau'\left(\alpha  y+(1-\alpha ) x\right)
 & \alpha\tau'\left(\alpha y+(1-\alpha  x\right)
\end {array} \right]=\left[ \begin {array}{cc} 0&1\\ \noalign{\medskip}\pm 2\left( 1-\alpha \right)
 & \pm 2\alpha
\end {array} \right],
\]
with $+$ sign for $(x,y)\in A_1$, the region below line $L$ and $-$ sign for $(x,y)\in A_2$, the region above line $L$.     
Similarly, when we consider the inverse branches $G_1^{-1}$ and $G_2^{-1}$ , 
 we have two Jacobian matrices:
\[B_{\pm}=DG^{-1}=\left[ \begin {array}{cc} \frac{-\alpha}{1-\alpha}&\pm \frac{1}{2(1-\alpha)}\\ \noalign{\medskip}1
 & 0
\end {array} \right].
\]

We now construct  invariant cones of directions in the tangent spaces  as in \cite{Ry}.

For $A_{\pm}$, we consider the direction vector in the form $(u,1)$. Then,
\begin{equation*}\begin{split}\left[ \begin {array}{cc} 0&1\\ \noalign{\medskip}\pm 2\left( 1-\alpha \right)
 & \pm2\alpha
\end {array} \right] \left[ \begin {array}{c} u\\ \noalign{\medskip}1
\end {array}\right]
&= \left[ \begin {array}{c} 1\\ \noalign{\medskip}\pm 2u\left( 1-\alpha \right)\pm 2\alpha
\end {array}\right]\\
&=(\pm 2u\left( 1-\alpha \right)\pm 2\alpha)\left[ \begin {array}{c} \frac{1}{\pm 2u\left( 1-\alpha \right)\pm 2\alpha}\\ \noalign{\medskip}1
\end {array}\right].
\end{split}
\end{equation*}
Let 
\be{S_\pm(u)=\frac{1}{\pm 2u\left( 1-\alpha \right)\pm 2\alpha},}
 be the corresponding transformation on directions.

For $B_{\pm}$, we consider the direction vector in the form $(1,v)$. Then,
\begin{equation*}\begin{split}\left[ \begin {array}{cc} \frac{-\alpha}{1-\alpha}&\pm \frac{1}{2(1-\alpha)}\\ \noalign{\medskip}1
 & 0
\end {array} \right] \left[ \begin {array}{c} 1\\ \noalign{\medskip}v
\end {array}\right]
&= \left[ \begin {array}{c} \frac{-\alpha}{1-\alpha}\pm \frac{v}{2(1-\alpha)}\\ \noalign{\medskip}1
\end {array}\right]\\
&=(\frac{-\alpha}{1-\alpha}\pm \frac{v}{2(1-\alpha)})\left[ \begin {array}{c} 1\\ \noalign{\medskip}\frac{1}{\frac{-\alpha}{1-\alpha}\pm \frac{v}{2(1-\alpha)}}
\end {array}\right].
\end{split}
\end{equation*}
Let 
\be{T_\pm(v)=\frac{1}{\frac{-\alpha}{1-\alpha}\pm \frac{v}{2(1-\alpha)}}=\frac {2(1-\alpha)}{-2\alpha\pm v},}
 be the corresponding transformation on directions.

\begin{figure}[tbp] 
  \centering
  \includegraphics[bb=0 -1 655 378,width=3.92in,height=2.27in,keepaspectratio]{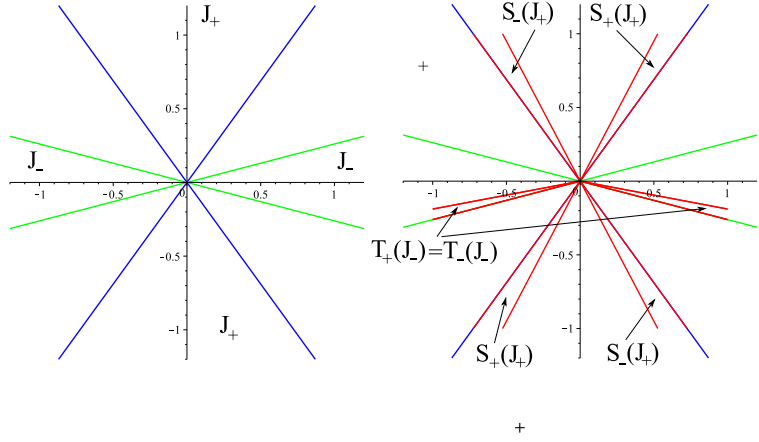}
  \caption{Invariant cones $J_+$ and $J_-$ and their images for $\alpha=0.82$.}
  \label{fig:invariant_cones}
\end{figure}

\begin{lemma}\label{inv_cones} (Rychlik \cite{Ry}, Lemma 3)\ 
Let $\theta_0=\alpha-\sqrt{\alpha^2+2\alpha-2}$, $J_+=\left\{u\in \mathbb{R}\big| |2u\left( 1-\alpha \right)|\leq\theta_0\right\}$, $J_-=\left\{v\in \mathbb{R}\big| |v|\leq\theta_0\right\}$. Then, $J_+$
and $J_-$ are $S_\pm-$ and $T_\pm-$invariant, respectively.
\end{lemma}
\begin{proof}
First, note that $\theta_0<\alpha$. We will prove the case of $S_+$. The case for $S_-$ is similar. It follows from $|2u\left( 1-\alpha \right)|\leq\theta_0$ that
\[2\alpha-\theta_0\leq 2\alpha+2u\left( 1-\alpha \right)\leq 2\alpha+\theta_0.\]
Thus,
\begin{equation}\label{ineq_contra}
|2(1-\alpha)S_+(u)|\leq \frac{2(1-\alpha)}{2\alpha-\theta_0}\leq \theta_0,
\end{equation}
where the last inequality follows from the definition of $\theta_0$.

Now we prove the case of $T_+$. The case of $T_-$ is similar. It follows from $|v|\leq\theta_0$ that
\[-3\alpha+\sqrt{\alpha^2+2\alpha-2}\leq -2\alpha+v\leq -\alpha-\sqrt{\alpha^2+2\alpha-2}.\]
Thus,
\begin{equation}\label{ineq_expad}
\begin{split}|T_+(v)|
&= \bigg| \frac{2(1-\alpha)}{-2\alpha+v}\bigg|
=\frac{2(1-\alpha)}{\big| -2\alpha+v\big|}\\
&\leq\frac{2(1-\alpha)}{\alpha+\sqrt{\alpha^2+2\alpha-2}}
=\theta_0.
\end{split}
\end{equation}

\end{proof}

\begin{remark}
We also have $S_\pm(J_+)\subseteq \left\{u\in \mathbb{R}\big| |2u\left( 1-\alpha \right)|\geq \frac{2(1-\alpha)}{2\alpha+\theta_0} \right\}=\left\{u\in \mathbb{R}\big| |u|\geq \theta_1 \right\}$,
$T_\pm(J_-)\subseteq \left\{v\in \mathbb{R}\big| |v|\geq 2(1-\alpha)\theta_1 \right\}$,
where $\theta_1=\frac{1}{2\alpha+\theta_0}$.
\end{remark}

\begin{lemma}\label{Lemma:kappa}(Rychlik \cite{Ry}, Lemma 4)\ 
Let $\kappa=\frac{\alpha-\sqrt{\alpha^2+2\alpha-2}}{\alpha+\sqrt{\alpha^2+2\alpha-2}}$, which is less than 1 (actually it is less than 0.5 and decreasing with respect to $\alpha$). Then,
$\sup_{_{J_+}}|S_\pm'(u)|=\sup_{_{J_-}}|T_\pm'(v)|=\kappa$.
\end{lemma}
\begin{proof}
It follows from (\ref{ineq_contra}) that
\[|S'_\pm(u)|= 2(1-\alpha)S^2_\pm(u)\leq 2(1-\alpha)\left(\frac{\theta_0}{2(1-\alpha)}\right)^2=\kappa.\]

And, it follows from (\ref{ineq_expad}) that
\[|T'_\pm(v)|= \frac{1}{2(1-\alpha)}T^2_\pm(v)\leq \frac{1}{2(1-\alpha)}\theta^2=\kappa.\]
\end{proof}

Using Lemma \ref{inv_cones} and Lemma \ref{Lemma:kappa}, we see that for any sequence $(\eps_0,\eps_1,\dots)\in\{+,-\}^\infty$, we have
$|(T_{\eps_{n-1}}\circ T_{\eps_{n-2}}\circ\dots\circ T_{\eps_{1}}\circ T_{\eps_{0}})(J_-)|\le\kappa^n|J_-|$ so the set $$\bigcap_{n=1}^\infty (T_{\eps_{n-1}}\circ T_{\eps_{n-2}}\circ\dots\circ T_{\eps_{1}}\circ T_{\eps_{0}})(J_-) , $$
consists of exactly one point which can be expressed as a continued fraction:
$$\frac{2(1-\alpha)}{-2\alpha+\frac{\epsilon_02(1-\alpha)}{-2\alpha+
\frac{\epsilon_12(1-\alpha)}{-2\alpha+\frac{\epsilon_22(1-\alpha)}{-2\alpha+\cdots}}}}.$$
Similarly, for any sequence $(\eta_0,\eta_1,\dots)\in\{+,-\}^\infty$, we have
$|(S_{\eta_{n-1}}\circ S_{\eta_{n-2}}\circ\dots\circ S_{\eta_{1}}\circ S_{\eta_{0}})(J_-)|\le\kappa^n|J_-|$ so the set $$\bigcap_{n=1}^\infty (S_{\eta_{n-1}}\circ S_{\eta_{n-2}}\circ\dots\circ S_{\eta_{1}}\circ S_{\eta_{0}})(J_-) , $$
consists of exactly one point which can be expressed as a continued fraction:
$$\frac{\eta_0}{2\alpha+2(1-\alpha)\frac{\eta_1}{2\alpha+2(1-\alpha)\frac{\eta_2}{2\alpha+\cdots}}}.$$
They are both convergent since $\kappa<1$.

Now,  using the above construction we define invariant directions for $G$. For points $p\in  U^s=[0,1]^2\backslash \bigcup\limits_{n=0}^{\infty}G^{-n}(L)$, setting $\eps_i=+$ or $\eps_i=-$, depending
on whether $G^i(p)$ is below or above the line $L$, we obtain the invariant stable direction $v(p)\in J_-$,
\[v(p)=\frac{2(1-\alpha)}{-2\alpha+\frac{\epsilon_02(1-\alpha)}{-2\alpha+
\frac{\epsilon_12(1-\alpha)}{-2\alpha+\frac{\epsilon_22(1-\alpha)}{-2\alpha+\cdots}}}}.\]
To construct an invariant unstable direction for a point $p$ we have to use $G$-preimages of $p$. Since $G$ is not invertible the ``invariant" direction
 will depend on the chosen admissible past of the point $p$.  Some points have only one admissible past,
for example, for the fixed point $(2/3,2/3)$ the only admissible past is $(\dots,2,2,\dots,2,2)$ and
it has the unique well defined unstable direction. Other points have finite number or infinitely many
admissible pasts. The richest case happens when
the directions in the set of ``invariant" directions form a Cantor set, namely the attractor of the Iterated Function System $\{S_+,S_-\}$.
For a point $p\in U^u=[0,1]^2\backslash \bigcup\limits_{n=0}^{\infty}G^{n}(L)$ with specified past $(\dots,k_{n-1},\dots,k_1,k_0)\in\{1,2\}^\infty$ we choose
$\eta_i=+$  when $k_i=1$ or $\eta_i=-$ when $k_i=2$, and obtain the invariant stable direction $u(p)\in J_+$,
\[u(p):=\frac{\eta_0}{2\alpha+2(1-\alpha)\frac{\eta_1}{2\alpha+2(1-\alpha)\frac{\eta_2}{2\alpha+\cdots}}}.\]

\begin{figure}[tbp] 
  \centering
  \includegraphics[bb=0 -1 309 306,width=3.92in,height=3.89in,keepaspectratio]{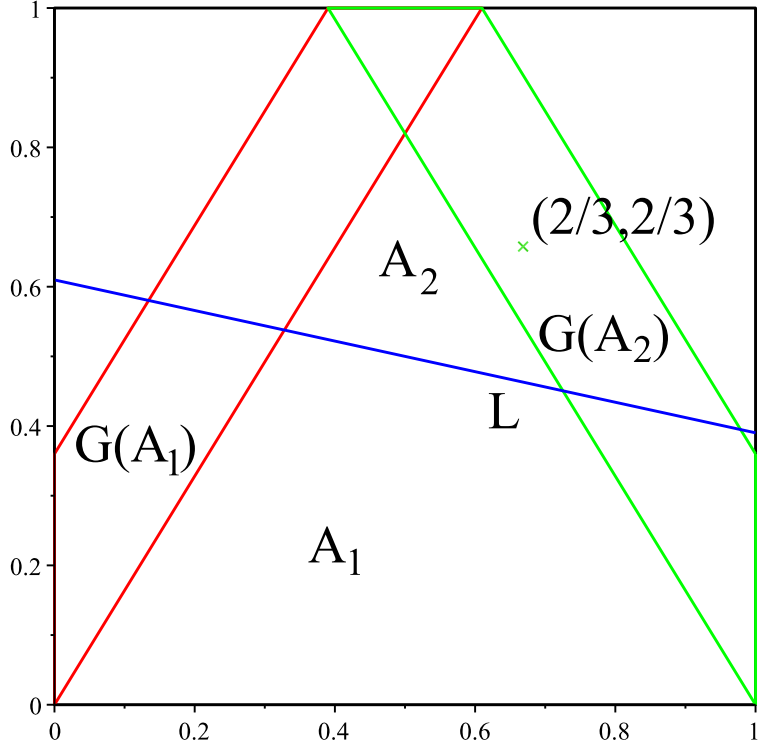}
  \caption{Partition line $L$, regions $A_1, A_2$ and their images, fixed point $(2/3,2/3)$ for $\alpha=0.82$.}
  \label{fig:partition_images}
\end{figure}

We now compute $\lambda^s(p)$ and $\lambda^u(p)$, which represent the rates of change on the length along
directions of $E^s$ and $E^u$, respectively. For directions in $E^u$ the rate is independent of the chosen past of the point.

\begin{lemma}\label{direc_rate} (Rychlik \cite{Ry}, Lemma 5)
\[\lambda^s(p)=|v(p)|\frac{h_1(p)}{h_1(G(p))}, \ \ \lambda^u(p)=|u(p)|\frac{h_2(p)}{h_2(G^{-1}(p))},\]
where,
\[h_1(p)=\frac{1}{\sqrt{(v(p))^2+1}},\ \ h_2(p)=\frac{1}{\sqrt{(u(p))^2+1}}.\]
\begin{proof}
\begin{equation*}\begin{split}DG(p)\left[ \begin {array}{c} 1\\ \noalign{\medskip}v(p)
\end {array}\right]&=\left[ \begin {array}{cc} 0&1\\ \noalign{\medskip}\pm 2\left( 1-\alpha \right)
 & \pm2\alpha
\end {array} \right] \left[ \begin {array}{c} 1\\ \noalign{\medskip}v(p)
\end {array}\right]\\
&= \left[ \begin {array}{c} v(p)\\ \noalign{\medskip}\pm 2\left( 1-\alpha \right)\pm2\alpha v(p)
\end {array}\right]
=v(p)\left[ \begin {array}{c} 1\\ \noalign{\medskip}\pm \frac{2\left( 1-\alpha \right)}{v(p)}\pm2\alpha
\end {array}\right]\\
&=v(p)\left[ \begin {array}{c} 1\\ \noalign{\medskip}T_{\pm}^{-1}(v(p))\end {array}\right]=v(p)\left[ \begin {array}{c} 1\\ \noalign{\medskip}v(G(p))
\end {array}\right].
\end{split}
\end{equation*}
Thus,
\[\lambda^s(p)=\frac{\|v(p)(1, v(G(p)))\|}{\|(1, v(p))\|}=|v(p)|\frac{h_1(p)}{h_1(G(p))}.\]
Similarly,
\begin{equation*}\begin{split}DG^{-1}(p)\left[ \begin {array}{c} u(p)\\ \noalign{\medskip}1
\end {array}\right]&=\left[ \begin {array}{cc} \frac{-\alpha}{1-\alpha}&\pm \frac{1}{2(1-\alpha)}\\ \noalign{\medskip}1
 & 0
\end {array} \right]  \left[ \begin {array}{c} u(p)\\ \noalign{\medskip}1
\end {array}\right]\\
&= \left[ \begin {array}{c} \frac{-\alpha}{1-\alpha}u(p)\pm \frac{1}{2(1-\alpha)}\\ \noalign{\medskip}u(p)
\end {array}\right]=u(p)\left[ \begin {array}{c} \frac{-\alpha}{1-\alpha}\pm \frac{1}{2(1-\alpha)u(p)}\\ \noalign{\medskip}1
\end {array}\right]\\
&=u(p)\left[ \begin {array}{c} S_\pm ^{-1}( u(p))\\ \noalign{\medskip}1
\end {array}\right]=u(p)\left[ \begin {array}{c} u(G^{-1}(p))\\ \noalign{\medskip}1
\end {array}\right],
\end{split}
\end{equation*}
and thus,
\[\lambda^u(p)=\frac{\|u(p)(u(G^{-1}(p)), 1)\|}{\|(u(p), 1)\|}=|u(p)|\frac{h_2(p)}{h_2(G^{-1}(p))}.\]
\end{proof}
\end{lemma}

We need the conditions that both $\theta_0$ and $\frac{\theta_0}{2(1-\alpha)}$ are less than $1$,
which hold since $\alpha\in(\frac{3}{4}, 1)$.

Now we present a proposition analogous to Proposition 5 in Rychlik \cite{Ry}.
\begin{proposition}\label{lambdas} 
Let $\lambda_+=\frac{\theta_0}{2(1-\alpha)}$, $\lambda_-=\theta_0$. Then both $\lambda_+, \lambda_-\in(0,1)$. And there
exists a constant $C>0$ such that $|\lambda^s_n(p)|\leq C \lambda^n_-$ if $p\in U^s$, $|\lambda^u_n(p)|\leq C \lambda^n_+$ if $p\in U^u$.
\end{proposition}
\begin{proof}
Using Lemma \ref{inv_cones} and the invariant sets $J_+$ and $J_-$, it follows that $h_1$ and $h_2$ are bounded, i.e. there exists numbers $c_1$ and $c_2$ such that $0<c_1\leq h_i\leq c_2$,
 $i=1,2$. Thus, by Lemma \ref{direc_rate}
\begin{equation*}\begin{split}
\lambda^s_n(p)&=\lambda^s(G^{n-1}(p))\cdot\lambda^s(G^{n-2}(p))\cdots\lambda^s(G(p))\lambda^s(p)\\
&=|v(G^{n-1}(p))|\frac{h_1(G^{n-1}(p))}{h_1(G^n(p))}\cdot|v(G^{n-2}(p))|\frac{h_1(G^{n-2}(p))}{h_1(G^{n-1}(p))}\cdots\\
&\quad\cdot|v(G(p))|\frac{h_1(G(p))}{h_1(G^2(p))}\cdot|v(p)|\frac{h_1(p)}{h_1(G(p))}\\
&=|v(G^{n-1}(p))|\cdot|v(G^{n-2}(p))|\cdots\cdot|v(p)|\frac{h_1(p)}{h_1(G^n(p))}\\
&\leq \frac{c_2}{c_1}\theta_0^n.
\end{split}
\end{equation*}
Similarly, we have
\[\lambda^u_n(p)\leq \frac{c_2}{c_1}\left(\frac{\theta_0}{2(1-\alpha)}\right)^n.\]

\end{proof}


Let $\mathcal P=\mathcal P^{(1)}$ be the partition of the square $[0,1]^2$ into the regions of definition of the map $G$, i.e.,
$A_1=\{(x,y): \alpha y+(1-\alpha)x\le 1/2\}$ and $A_2=\{(x,y): \alpha y+(1-\alpha)x\ge 1/2\}$. These regions intersect, but the intersection is a negligible set both in a measure-theoretic and topological sense.
We define $\mathcal P^{(n)}= \mathcal P\bigvee G^{-1}(\mathcal P)\bigvee G^{-2}(\mathcal P)\bigvee\dots\bigvee G^{n-1}(\mathcal P)$.
$\mathcal P^{(n)}$ is the defining partition  for the map $G^n$.

Let $L$ denote the partition line $$L=\{p=(x,y): \alpha y+(1-\alpha) x=1/2\}.$$

\begin{lemma}\label{cover} (Rychlik \cite{Ry}, Lemma 8) For every $N\ge 1$ there is an open cover $\mathcal U_N$ of the unit square such that every element of $\mathcal U_N$ intersects no more than $2N$ elements of $\mathcal P^{(N)}$.
\end{lemma}
The proof is exactly the same as in \cite{Ry}.

\begin{proposition} \label{Rych7}(Rychlik \cite{Ry}, Proposition 7) There exist constants $F>0$ and $0<r<1$ such that for any segment $I$ with the direction from the unstable cone $J_+$ we have 
\begin{equation}\label{ineq:P7}
\Gamma_n(I)=\sum_{J\in \mathcal P^{(n)}|I} \frac{|J|}{|G^n(J)|}\le F(r^n+|I|),
\end{equation}
where $|\cdot|$ denotes the length of the segment.
\end{proposition}
\begin{proof}
The proof follows closely the proof from \cite{Ry}. We choose $N$ in such a way that 
$$r_0=2N C(\lambda_+)^N<1,$$
 where $C$ and $ \lambda_+$ are from Proposition \ref{lambdas}. Let $\eps_0$ be the Lebesgue constant of the cover $\mathcal U_N$ from Lemma \ref{cover}. Let us define
\begin{equation}
\gamma_n=\sum_{J\in \mathcal P^{(nN)}|I} \frac{|J|}{|G^{nN}(J)|},\ \ \ n=1,2,\dots
\end{equation}
We will show that 
\begin{equation}\label{inequality_gamma}
\gamma_{n+1}\le r_0\gamma_n+\frac 1{\eps_0}R_0|I|,
\end{equation}
where $$R_0=\sup_I \gamma_1=\sup_I \sum_{J\in \mathcal P^{(N)}|I} \frac{|J|}{|G^{N}(J)|},$$ and $I$ is any segment with the direction from the unstable cone $J_+$.

Let $J\in \mathcal P^{(nN)}|I$. Either $|J|<\eps_0$ or $|J|\ge\eps_0$. In the first case
$\mathcal P^{((n+1)N)}|J$ consists of not more that $2N$ elements (Lemma\ref{cover}) as the partition
$\mathcal P^{((n+1)N)}$ is obtained from $\mathcal P^{(nN)}$ in $N$ steps. Thus,
\begin{equation}\label{est1}\begin{split}
\sum_{J'\in \mathcal P^{((n+1)N)}|J} \frac{|J'|}{|G^{(n+1)N}(J')|}&\le 2N \max_{J'} \frac{|J'|}{|G^{(n+1)N}(J')|}\\
&= 2N \ \max_{J'}\left(\frac{|J'|}{|G^{nN}(J')|}\ \frac{|G^{nN}(J')|}{|G^{(n+1)N}(J')|}\right)\\
&= 2N\ \frac{|J|}{|G^{nN}(J)|}\ \max_{J'} \frac{|G^{nN}(J')|}{|G^{(n+1)N}(J')|}\\
&\le 2N \frac{|J|}{|G^{nN}(J)|}\ C (\lambda_+)^N=r_0 \frac{|J|}{|G^{nN}(J)|}.
\end{split}
\end{equation}
We have used the fact that $J'\subset J$ and $G^{nN}$ is a linear transformation on $J$, so the expansion rate is uniform on $J$.

In the second case we have
\begin{equation}\label{est2}\begin{split}
\sum_{J'\in \mathcal P^{((n+1)N)}|J} \frac{|J'|}{|G^{(n+1)N}(J')|}&= \sum_{J'\in \mathcal P^{((n+1)N)}|J} \left(\frac{|J'|}{|G^{nN}(J')|}\ \frac{|G^{nN}(J')|}{|G^{(n+1)N}(J')|}\right)\\
&= \frac{|J|}{|G^{nN}(J)|}\  \sum_{J'\in \mathcal P^{((n+1)N)}|J} \frac{|G^{nN}(J')|}{|G^{(n+1)N}(J')|}\\
&\le \frac{|J|}{|G^{nN}(J)|}R_0\le R_0\frac 1{\eps_0}{|J|}.
\end{split}
\end{equation}
We used again the linearity of $G^{nN}$ on $J$. Moreover
$$\sum_{J'\in \mathcal P^{((n+1)N)}|J} \frac{|G^{nN}(J')|}{|G^{(n+1)N}(J')|}\le
\sum_{K\in \mathcal P^{(N)}|G^{nN}(J)} \frac{|K|}{|G^{N}(K)|}\le R_0, $$
since intervals $G^{nN}(J')$ are elements of $\mathcal P^{(N)}|G^{nN}(J)$ and $G^{nN}(J)$ has the direction 
from $J_+$. Also $|G^{nN}(J)|> |J|>\eps_0$.

Summing up (\ref{est1}) and (\ref{est2}) over all  $J\in \mathcal P^{(nN)}|I$, we obtain
\begin{equation}\begin{split}
\sum_{J'\in \mathcal P^{((n+1)N)}|I} \frac{|J'|}{|G^{(n+1)N}(J')|}&\le \sum_{J\in \mathcal P^{(nN)}|I} \left(r_0 \frac{|J|}{|G^{nN}(J)|}+R_0\frac 1{\eps_0}{|J|}\right)\\
&= r_0 \gamma_n+ R_0\frac 1{\eps_0} |I|,
\end{split}
\end{equation}
and 
(\ref {inequality_gamma}) is proved.
To obtain inequality (\ref{ineq:P7}) from (\ref{inequality_gamma}) we proceed as follows.
Since $\gamma_1\le R_0$ by definition, the inequality (\ref{inequality_gamma}) implies
\begin{equation}
\gamma_n\le r_0^nR_0 +\frac {R_0}{\eps_0(1-r_0)}|I|, \ \ n=1,2,\dots ,
\end{equation}
or, using capital gamma notation
$$\Gamma_{nN}(I)\le r_0^nR_0 +\frac {R_0}{\eps_0(1-r_0)}|I|, \ \ n=1,2,\dots. $$
Let us define $$\bar R_i=\sup_I \Gamma_i(I)=\sup_I\sum_{J\in \mathcal P^{(i)}|I} \frac{|J|}{|G^{i}(J)|},\ \ \ i=1,2,\dots,N ,$$
where $\sup$ is taken over all segments with the direction in the expanding cone $J^+$. Of course $\bar R_N=R_0$. Let $R=\max\{\bar R_1,\bar R_2,\dots,\bar R_N\}$.
Let us consider arbitrary $n\ge 1$ and represent it as $n=k\cdot N+\ell$, $0<\ell\le N$.
Similarly as above,using in all considerations  as the initial partition $\mathcal P^{(\ell)}|I$ instead of
 $\mathcal P^{(N)}|I$, we can prove that
$$\Gamma_n(I)\le r_0^k \bar R_\ell+ \frac {\bar R_\ell}{\eps_0(1-r_0)}|I|.$$
 To make these estimates independent of $\ell$ we can write
$$\Gamma_n(I)\le r_0^k  R+ \frac { R}{\eps_0(1-r_0)}|I|.$$
Now, let $r=(r_0)^{1/N}$ and $F=\max\{\frac{R}{r^{N-1}}, \frac { R}{\eps_0(1-r_0)}\}$.
We obtain inequality (\ref{ineq:P7}).
\end{proof}
We define $\mathcal P^{-}=\bigvee_{n=0}^\infty  G^{-n}(\mathcal P)$. Elements of $\mathcal P^{-}$ are either segments with
 direction from the stable cone or points. Let $\xi(p)\in \mathcal P^{-}$ denote an element of $\mathcal P^{-}$ containing point $p$.

\begin{lemma}\label{lem:9} (corresponds to Lemma 9 of \cite{Ry}) Let 
\begin{equation}\label{def:D^s}
D^s(\delta)=\{p\in [0,1]^2:{\rm dist} (G^np, L)\ge \delta \lambda^s_n(p) , {\rm for}\ n=0,1,2,\dots\}.
\end{equation}
For every $p\in D^s(\delta)$ the distance from $p$ to the endpoints of $\xi(p)$ is not smaller than $\delta$. 
In particular, $|\xi(p)|\ge 2 \delta$.
\end{lemma}
\begin{proof} Assume that the distance from $p$ to one of the endpoints of $\xi(p)$ called $q$
 is ${\rm dist}(p,q)<\delta$. Since endpoints of elements $\xi$ belong to preimages $G^{-n}(L)$,
there is an integer $k\ge 0$ such that $q\in G^{-k}(L)$. Then, 
 $${\rm dist}(G^k p,L)\le {\rm dist}(G^k p,G^k q)\le \lambda^s_k(p){\rm dist}(p,q)<\delta \lambda^s_k(p) ,$$
which contradicts $p\in D^s(\delta)$.
\end{proof}
\begin{lemma} \label {Rych10} (corresponds to Lemma 10 of \cite{Ry}) Let $(\lambda_n)=(\lambda_n)_{n=0}^\infty$ be a sequence of positive numbers such that
$Z=\sum_{n=0}^\infty \lambda_n <+\infty$. Let
\begin{equation}\label{def:D^slambda}
D^s(\delta,(\lambda_n))=\{p\in [0,1]^2:{\rm dist} (G^np, L)\ge \delta \lambda_n , {\rm for}\ n=0,1,2,\dots\}.
\end{equation}
Let $I$ be a segment with direction from unstable cone. Then, there is a constant $A_1$ such that $|I\setminus D^s(\delta,(\lambda_n))|\le A_1\cdot Z\cdot\delta.$
\end{lemma}
\begin{proof} We follow closely Rychlik \cite{Ry}.
Let $$C(t)=\{q: \dis(q,L)\le t\},$$ where $t\ge 0$.
Let $p\in I\setminus D^s(\delta,(\lambda_n))$. There exists $n\ge 0$ such that $\dis(G^np,L)<\delta \lambda_n$.
Let $J\in \mathcal P^{(n)}|I$ be the subinterval containing point $p$. Then, $G^n p$ belongs to the interval $G^n J$ such that
$$| G^n J \cap C(\delta \lambda_n)\}|\le A_0 \cdot \delta\lambda_n , $$
for some constant $A_0$ independent of $\delta$ and $n$, as $G^nJ$ has a direction from the expanding cone and thus, the angle between $G^n J$ and line $L$ is bounded away from 0.
Thus, $p\in J\cap G^{-n}(C(\delta \lambda_n))$ and 
$$|J\cap G^{-n}(C(\delta \lambda_n))|\le \frac {A_0 \cdot \delta\lambda_n}{|G^n J|}\cdot |J|.$$
By Proposition \ref{Rych7}, this gives
$$|I\cap G^{-n}(C(\delta \lambda_n))|\le {A_0 \cdot \delta\lambda_n}\cdot F(r^n+|I|)\le A_0 F(1+{\rm diam} ([0,1]^2))\delta\lambda_n .$$
Summing up over all $n$, we obtain $$|I\setminus D^s(\delta,(\lambda_n))|\le A_0F(1+\sqrt{2})\delta Z .$$
The Lemma is proved with $A_1=A_0F(1+\sqrt{2})$.
\end{proof}
\begin{corollary}\label{cor1} For any interval $I$ with the direction from the expanding cone we have 
$$|I\setminus D^s(\delta)|\le A_2\cdot \delta ,$$
where $A_2= A_1\sum_{n=0}^\infty C \lambda_{-}= A_1 C /(1-\lambda_{-})$.
\end{corollary}
\begin{proof} Let $\lambda_n=C\lambda_{-}^n$, $n=0,1,2,\dots$ Since $\lambda^s_n \le C\lambda_{-}^n$ we have $D^s(\delta)\supset D^s(\delta,(\lambda_n))$.
This proves the claim.
\end{proof}

Let $\nu$ denote the normalized Lebesgue measure on $[0,1]^2$.

\begin{corollary} The set $\tilde D^s=\bigcup_{\delta>0} D^s(\delta)$ is of full $\nu$-measure in $[0,1]^2$. Moreover
$$\nu([0,1]^2\setminus D^s(\delta))\le A_2\cdot \delta.$$
\end{corollary}
\begin{proof}
Follows by Corollary \ref{cor1} and Fubini's Theorem.
\end{proof}

\begin{figure}[tbp] 
  \centering
  \includegraphics[bb=0 -1 311 329,width=3.92in,height=4.16in,keepaspectratio]{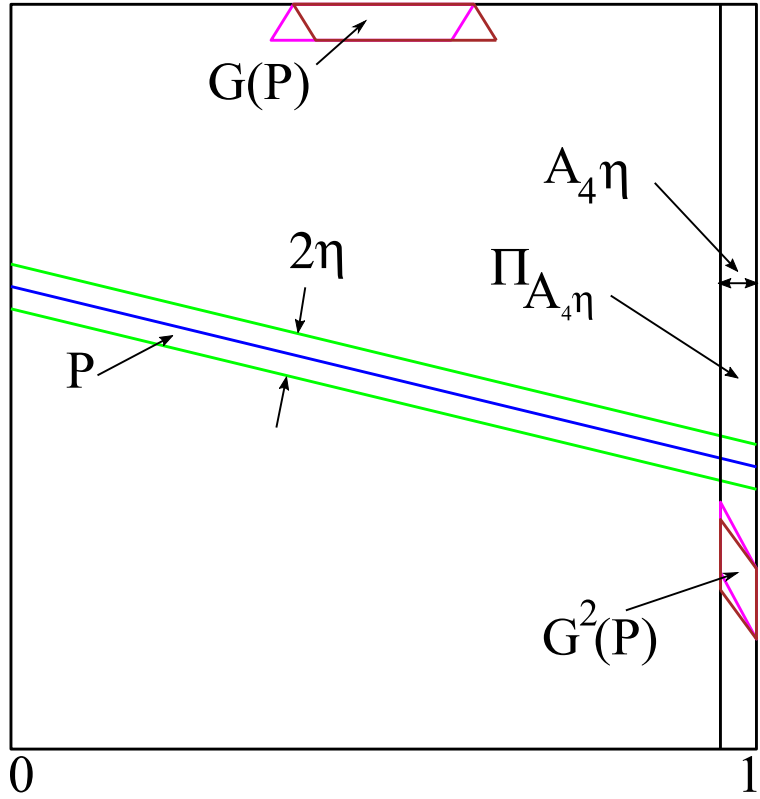}
  \caption{Two images of a neighbourhood of the partition line $L$. Both $G(P)$ and $G^2(P)$ are unions of
two parallelograms.}
  \label{fig:image_partition_line}
\end{figure}

Let us consider the function $1/D(p)$ where $D(p)=|\xi(p)|$. We will prove that it is integrable.

\begin{proposition}(corresponds to  Proposition 8  of \cite{Ry})\label{prop:8} There is a constant $A_3>0$ such that  for an arbitrary $\delta>0$,
$$\nu(\{p\in [0,1]^2: D(p)<\delta\})\le A_3 \delta^2.$$
\end{proposition}
\begin{proof}
If $p\in \{p\in [0,1]^2: D(p)<\delta\}$, then $\dis(G^np,L)<\delta\lambda^s_n(p) $ at least for two $n_1<n_2$ since both ends of $\xi(p)$ have to be trimmed (Lemma \ref{lem:9}).
This means that $p$ is less that $\delta$ close to  preimage $G^{-n_1}(L)$ and $\dis(G^{n_1}p,L)<\delta\lambda^s_{n_1}(p)=\eta$. 
 Then, $$G^{{n_1}+2}p\in \Pi_{A_4\eta}=\{(x,y): 1-A_4\eta\le x\le 1, 0\le y\le 1\},$$
for some constant $A_4>0$ independent of $\delta$ and $n_1$. See Figure \ref{fig:image_partition_line}.
Also,  $G^{{n_1}+2}p\in [0,1]^2\setminus D^s(\eta)$ since $\dis(G^{n_2}p,L)<\delta\lambda^s_{n_2}(p)<\eta $, and $G^{{n_1}+1}p$ is far from the line $L$.
Let $n=n_1+2$. We have $G^n p\in \Pi_{A_4\eta}\setminus D^s(\eta)$. Since the vertical direction is in the expanding cone, by Corollary \ref{cor1} and Fubini's Theorem, we have
$$\nu(\Pi_{A_4\eta}\setminus D^s(\eta))\le A_2\cdot A_4\cdot \eta^2 . $$
Thus, $$\nu(\{p\in [0,1]^2: D(p)<\delta\})\le \sum_{n=0}^\infty \nu(G^{-n}(\Pi_{A_4\eta}\setminus D^s(\eta))\le  \sum_{n=0}^\infty  A_2\cdot A_4\cdot \eta^2 (2\cdot {\rm Jac}^{-1}(\alpha))^n, $$
where ${\rm Jac}(\alpha)=2(1-\alpha)$ is the Jacobian of both $G_1$ and $G_2$. We need the multiplier $2$ because $G$ is a 2 to 1 map.(This is different from the Lozi map studied in  \cite{Ry}.)
By Lemma \ref{direc_rate} and Proposition \ref{lambdas} we have
$$\lambda^s_n(p)\le C\lambda_{-}^n,$$
where
$$\lambda_{-}=\alpha-\sqrt{\alpha^2+2\alpha-2}.$$
We have 
$$\nu(\{p\in [0,1]^2: D(p)<\delta\})\le A_2\cdot A_4\cdot C^2 \cdot \delta^2\cdot \sum_{n=0}^\infty \left(\frac {2(\alpha-\sqrt{\alpha^2+2\alpha-2})^2} {2(1-\alpha)} \right)^n.$$
It can be easily proved that for $3/4<\alpha<1$ we have
$\frac {(\alpha-\sqrt{\alpha^2+2\alpha-2})^2} {(1-\alpha)}<1$. Thus, the series converges to some
constant $A(\alpha)$, and setting $A_3=A_2\cdot A_4\cdot C^2 \cdot A(\alpha)$ completes the proof of the proposition.
\end{proof}

\begin{figure}[h] 
  \centering
  \includegraphics[bb=0 -1 671 337,width=3.92in,height=1.97in,keepaspectratio]{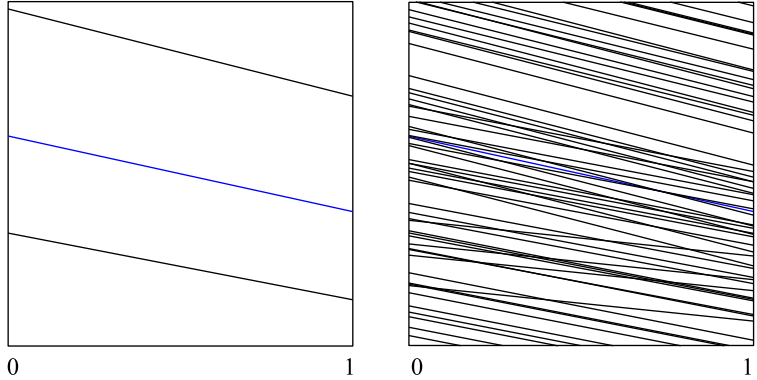}
  \caption{Partitions $\mathcal P^{(2)}$ and $\mathcal P^{(6)}$ for $\alpha=0.82$. }
  \label{fig:partitions}
\end{figure}

\begin{corollary}\label{integrable}(corresponds to Corollary 3 of \cite{Ry}) The function $p\mapsto 1/D^\beta(p)$ is integrable on $[0,1]^2$ for any  $\beta\in[1,2)$.
\end{corollary}
\begin{proof}
We will use the following identity for positive random variables
\be {E(X)=\int_0^\infty P(X>t)\, dt}
which can be found, e.g., in \cite{Bil}, page 275.
We have
\bes {&\int (1/D)^\beta d\nu=\int_0^\infty \nu(\{D^{-\beta}>\gamma\}) d\gamma\le 1+\int_1^\infty \nu(\{D^{-\beta}>\gamma\}) d\gamma\\
&\le 1+\int_1^\infty \nu(\{D<\gamma^{-1/\beta}\}) d\gamma\le 1+\int_1^\infty A_3\gamma^{-2/\beta} d\gamma
<+\infty}
\end{proof}

In the following proposition we will discuss the family of conditional measures $\{\nu_C\}_{C\in \mathcal P^{-}}$ of measure $\nu$ on elements of the partition $\mathcal P^{-}$.
The theory of conditional measures can be reviewed by referring to \cite{Sim} or \cite{Mah}.
Let $\{\ell_C\}_{C\in \mathcal P^{-}}$ be the family of one-dimensional Lebesgue measures on the elements of $\mathcal P^{-}$.

\begin{proposition}\label{prop:Ry9} (corresponds to Proposition 9 of \cite{Ry}) For almost every $C\in \mathcal P^{-}$, measure $\nu_C$ is absolutely continuous with respect to $\ell_C$ and 
the Radon-Nikodym derivative $\frac {d\nu_C}{d\ell_C}$ is constant on $C$, equal to $1/|C|$.
\end{proposition}

\begin{figure}[h] 
  \centering
  \includegraphics[bb=0 -1 592 156,width=3.92in,height=1.04in,keepaspectratio]{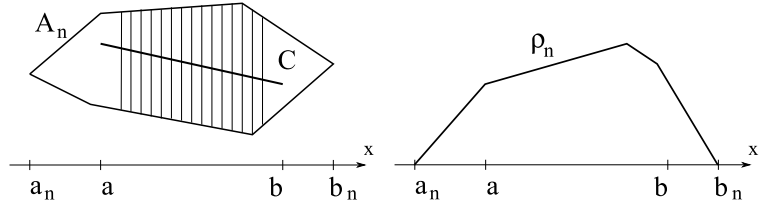}
  \caption{A polygon $A_n$ of the partition $\mathcal P^n$ and the density $\rho_n$.}
  \label{fig:Polygon}
\end{figure}

\begin{proof} We follow closely \cite{Ry}.  Let $A_n$ be the polygon of the partition $\mathcal P^n$, $n\ge 1$ containing $C\in \mathcal P^-$. See Figure \ref{fig:Polygon}.
Since $A_n$ is convex, the projection of the measure $\frac 1{\nu(A_n)} v_{|A_n}$ onto the $x$-axis is a measure absolutely continuous with respect to Lebesgue measure with density
$\rho_n$ which is positive on some interval $(a_n,b_n)$ and zero outside of this interval. Since $\rho_n(t)$ is proportional to the length of the intersection of the vertical line $x=t$
with the polygon $A_n$ the density $\rho_n$ is concave on $(a_n,b_n)$. We have $(a_{n+1},b_{n+1})\subset (a_n,b_n)$ and $a_n\to a$, $b_n\to b$ where $a$ and $b$ are the end points of the projection 
of $C$ onto the $x$-axis. Since $\rho_n(a_n)=\rho_n(b_n)=0$, $\int_{a_n}^{b_n}\rho_n=1$, and $\rho_n$ are concave the family $\{\rho_n\}_{n\ge 1}$ is uniformly bounded by $2/(b-a)$.
Since they are concave their variations are also uniformly bounded by $4/(b-a)$. By Helly's theorem (\cite{Bil}),  there exists a subsequence $\rho _{n_k}$ convergent to some density $\rho$ almost everywhere.
$\rho$ is concave as a limit of concave functions.
Projecting $\rho$ onto $C$ we obtain $\frac{d\nu_C}{d\ell_C}$ which is also concave. We will denote it again by $\rho$.

\begin{figure}[h] 
  \centering
  \includegraphics[bb=0 -1 348 239,width=3.92in,height=2.7in,keepaspectratio]{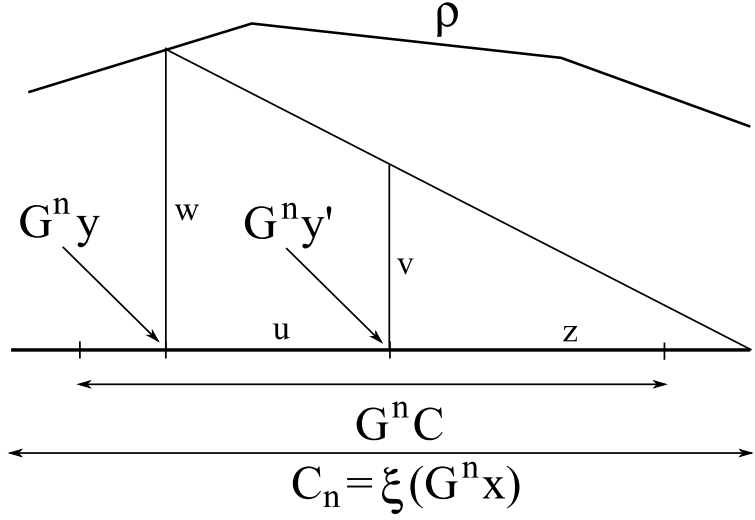}
  \caption{Using the concavity of $\rho_{|\xi(G^nx)}$.}
  \label{fig:concavity}
\end{figure}

To use Proposition \ref{prop:Ry3} we will prove that for almost every $x\in [0,1]^2$ and almost every $y,y'\in C(x)=C$, we have 
\be{\label{lim1}\lim_{n\to\infty}\frac {\rho(G^ny)}{\rho(G^ny')}=1.}
By Lemma \ref{Rych10} we assume that $\dis(G^nC,L)>\delta \lambda^n$, where $\delta>0$ and $\lambda\in(\lambda_{-},1)$. Now, we use the concavity of $\rho_{|\xi(G^nx)}$. See Figure \ref{fig:concavity}.
From the triangles in Figure  \ref{fig:concavity} we have $\frac w v=\frac{u+z} z= 1+\frac u z $. Thus,
$$\frac{\rho(G^ny)}{\rho(G^ny')}\le \frac w v= 1+\frac u z\le 1+\frac{\dis(G^n y,G^n y')}{\dis(G^n C, \partial(\xi(G^nx)))}\le 1+\frac {C\lambda_{-}^n}{\delta\lambda^n},$$
which goes to 1 as $n\to\infty$. Thus, $\limsup_{n\to\infty}\frac {\rho(G^ny)}{\rho(G^ny')}\le 1$. By symmetry, we obtain (\ref{lim1}).
By Proposition \ref{prop:Ry3} we have 
$$\frac{\rho(y)}{\rho(y')}=\prod_{k=0}^{\infty}\frac {g_G(G^ky)}{g_G(G^ky')}\frac{\lambda(G^ky)}{\lambda(G^ky')}.$$
Since the Jacobian of $G$ is constant and $G$ is piecewise linear and in particular linear on every $C\in\xi$, the right hand side of the above formula is constant.
\end{proof}


\section{Applying the Abstract Theorems to Transformations $G_\alpha$.}

We will use the notation introduced in Section \ref{Abstract}. Let $Y=[0,1]^2$, $S=G$ and $\nu$ be Lebesgue measure on $[0,1]^2$. 
The map $T$ is the factor map induced by $G$ on the space $X=[0,1]^2/\mathcal P^{-}$.

By formula (\ref{gTgS}) we have
$$g_T(x)=g_G(y)\frac{d((G_{|A})^{-1}_* \nu_{C(Tx)})}{d\nu_{C(x)} }(y),$$
for almost every $x$ and almost every $y$, where $A$ is an element of the partition $\mathcal P$.

\begin{lemma} We can rewrite $g_T$ as follows:   
\be{\label{g_T formula} g_T=\frac 1{{\rm Jac}_G}\lambda^s\frac D{D\circ T},}
 where $\lambda^s$ is defined in Lemma  \ref{direc_rate} and $D(x)=|C(x)|$.
\end{lemma}
\begin{proof} 
We can write
$$\frac{d((G_{|A})^{-1}_* \nu_{C(Tx)})}{d\nu_{C(x)} }=\frac {d((G_{|A})^{-1}_* \nu_{C(Tx)})}{d((G_{|A})^{-1}_* \ell_{C(Tx)})}\frac{d((G_{|A})^{-1}_* \ell_{C(Tx)})}{d\ell_{C(x)} }\frac {d\ell_{C(x)} }{d\nu_{C(x)} }.$$
In view of Proposition \ref{prop:Ry9} this gives the required formula for $g_T$.
\end{proof}
Since $g_T$ given by formula (\ref{g_T formula}) is very discontinuous we will replace it by considering instead of Lebesgue measure on $[0,1]^2$ an equivalent measure 
$\nu=\frac 1D \bar \nu$, where $\bar \nu $ is the Lebesgue measure. Then, we define $m=\frac 1d \bar m$, where $\bar m$ is the factor of the Lebesgue measure on $X=[0,1]^2/\mathcal P^{-}$.

\begin{proposition} (Rychlik \cite{Ry}, Proposition 10) If we apply the results of Section \ref{Abstract} to the measure $\nu=\frac 1D \bar \nu$, then 
\be{g_T=\frac 1{{\rm Jac}_G}\lambda^s .}
\end{proposition}
\begin{proof} Let $A\in \mathcal P^{-}$. By the definition of $g_T$  (\ref{def g_T}), we have
\bes{g_T&=\frac{d(T_*(\chi_A\cdot m))}{dm}\circ T=\frac{d(T_*(\chi_A\cdot\frac 1D\cdot \bar m))}{d(\frac 1D \cdot\bar m)}\circ T\\
&=\frac{\frac 1D \circ (T_{|A})^{-1}\cdot dT_*(\chi_A\bar m)}{\frac 1D d\bar m}\circ T\\
&=\frac{\frac{1}{D}}{\frac 1D\circ T}\cdot \frac 1{{\rm Jac}_G}\lambda^s\frac D{D\circ T}=\frac 1{{\rm Jac}_G}\lambda^s.
}
\end{proof}

We will now verify assumptions (I)--(IV).

\begin{lemma}\label{lemma:Ry12} (Rychlik \cite{Ry}, Lemma 12) Condition (I) is satisfied.
\end{lemma}

\begin{proof} Let $n\ge 1$ and $x_1, x_2\in B\in \beta^{(n)}$. We can treat $x_1,x_2$ as elements of
the Lebesgue  space $X$ and also as points in $[0,1]^2$ or elements of $\mathcal P^-$.
The points $G^k x_1$ and $G^k x_2$ are on the same side of the partition line $L$ for $k=1,2,\dots,n-1$. Since $\Jac_G$ is constant, we need only to find a universal constant $d$
such that
$$\frac 1d\le\frac {\lambda^s_n(x_1)}{\lambda^s_n(x_2)}\le d \, .$$
By Lemma \ref{direc_rate} we have
$\lambda^s(p)=|v(p)|\frac{h_1(p)}{h_1(G(p))}$, so 
$$\lambda^s_n(p)=|v(p)|\cdot |v(Gp)|\cdot\dots\cdot|v(G^{n-1}p)|\frac{h_1(p)}{h_1(G^n(p))}.$$
By Lemma \ref{Lemma:kappa} we have  $|v(Gp)-v(Gp')|\le \kappa|v(p)-v(p')|$, where $0<\kappa<1$.
Thus, $|v(G^kx_1)-v(G^kx_2)|\le \kappa^{n-k}|J_{-}|$ , for $k=1,2,\dots,n-1$.
Thus, there exists a constant $d_0$ such that
$$\exp\left(-d_0\kappa^{n-k}\right) \le \frac{\lambda^s(G^k x_1)}{\lambda^s(G^k x_2)}\le \exp\left(d_0\kappa^{n-k}\right) .$$
Then, for some constant $d_1$ (we have to include the fractions ${h_1(x_1)}/{h_1(G^n(x_1))}$ and ${h_1(x_2)}/{h_1(G^n(x_2))}$),  we obtain 
$$\exp\left(-d_1\sum_{k=0}^{n-1}\kappa^{n-k}\right) \le \frac{\lambda^s_n( x_1)}{\lambda^s_n( x_2)}\le \exp\left(d_1\sum_{k=0}^{n-1}\kappa^{n-k}\right) .$$
Letting $d=\exp(d_1/(1-\kappa))$, completes the proof.
\end{proof}

\begin{lemma} \label{lemma:Ry13} (Rychlik \cite{Ry}, Lemma 13) Conditions (II) and (IV) are satisfied for some iteration $T^N$, $N\ge 1$. 
\end{lemma}

\begin{proof} Condition (IV) is satisfied  because $g_n=(\Jac_G)^{-n}\cdot \lambda^s_n\le (\Jac_G)^{-n}\cdot(C\theta_0^n)=C(\theta_0/\Jac_G)^n$.
For large $n$, $g_n<1$, since $\theta_0/\Jac_G<1$.
We used Proposition \ref{lambdas} and $\theta_0=\alpha-\sqrt{\alpha^2+2\alpha-2}$, $\Jac_{G_\alpha}=2(1-\alpha)$. For $3/4<\alpha<1$,
we have $\theta_0/\Jac_G<1$.

Now, we will prove that condition (II) is satisfied for some iterate of $T$. The proof is similar to that of Proposition \ref{Rych7}.
Let  $N$ be such that $r_0=2NC(\lambda_+)^N<1$ and let $\eps_0$ be the Lebesgue constant of the cover $\mathcal U_N$ of Lemma \ref{cover}.
By Proposition \ref{lambdas}, $\lambda_+=\theta_0/\Jac_G$ so $g_n\le C\cdot \lambda_+$ for $n\ge 1$.

Let $B\in\beta^{(N)}$ and $A=\pi^{-1}(B)$. Then, $T^N(B)=\pi(G^NA)$. $G^NA$ is a convex polygon such that
$$\partial(G^N(A))\subset \cup_{k=0}^N G^k(L). $$
Except for $L$ itself and the first image $G(L)\subset\{y=1\}$, all subsequent images $G^k(L)$, $k\ge 2$, consist of segments with directions from the unstable cone $J_+$.
Also, all images of the sides of $[0,1]^2$ have this property.
We can assume $\eps_0$ is much smaller than the distance between $L$ and $G(L)$.
There are two possibilities:

(1) $\diam(G^N A)\ge\eps_0$. Then, $G^NA$ contains a segment $I$ with the unstable direction (from $J_+$) of length $A_5\cdot\eps_0$, for some $A_5\le 1$.
If all sides of $G^NA$ belong to $\cup_{k=2}^N G^k(L)$, i.e., they have directions from $J_+$, then obviously $G^NA$ contains a segment $I$ with the direction from $J_+$ of length $\eps_0$.
If one of the sides belongs to $L$ and another to $G(L)$, then $G^NA$ also contains such segment since $\eps_0$ is small.
If only one side of $G^NA$ belongs to $L$ or $G(L)$, then the worst case scenario is a triangle with two remaining sides with directions from $J_+$. Since the angle between directions from $J_+$
and $L$ or $G(L)$ is separated from zero,  $G^NA$ contains a segment $I$ with the  direction from $J_+$ of length $A_5\cdot\eps_0$, for some $A_5\le 1$.

By Corollary \ref{cor1}, for arbitrary $\delta>0$, $|I\setminus D^s(\delta)|\le A_2\delta$. The set $\bar A=\pi^{-1}\pi(G^NA)=\pi^{-1}(T^N B)$ has measure larger than
$A_4^{-1}(1-A_2\delta)\cdot A_5\cdot\eps_0$, where $A_4$ will be found in the following Lemma \ref{Ry14}. So, we put $\delta=\frac 12 A_2^{-1}$ and $\eps=A_4^{-1}(1-A_2\delta)\cdot A_5\cdot\eps_0=\frac 12A_4^{-1}\cdot A_5\cdot\eps_0$.
Then, $\nu(\bar A)>\eps$ and $m(T^NB)=\nu(\bar A)>\eps$, since $m=\pi_*\nu$.

(2) $\diam(G^N A)<\eps_0$. Then, $T^NB$ is contained in no more than $2N$ elements of $\beta^N$ and
$$\sum_{B'\in\beta^N(T^NB)} \sup_{B'} g_N\le (2N)(C\lambda_+^N)=r_0 .$$
Thus, condition (II) is satisfied for $T^N$ with $r=r_0$.
\end{proof}

\begin{lemma}\label{Ry14} (Rychlik \cite{Ry}, Lemma 14) Let $I$ be a segment with the direction from $J_+$ and let $\ell_I$ be the Lebesgue measure on $I$.
Then, the measure $\pi_*(\ell_I)$ is absolutely continuous with respect to $m$ and
\be{\label {density on I delta} \frac {d\pi_*(\ell_I)}{dm} (x)= \frac 1 {\sin \angle(I, x)},}
for $x\in X$, where $\angle(I, x)$ is the angle between segment $I$ and segment $C(x)\in \mathcal P^{-}$.
In particular, for some $A_4>0$ we have
\be{\label{ineq: A_4}\frac 1{A_4}\le \frac {d\pi_*(\ell_I)}{dm} \le A_4 .}
\end{lemma}

\begin{figure}[h] 
  \centering
  \includegraphics[bb=0 -1 306 357,width=3.92in,height=4.58in,keepaspectratio]{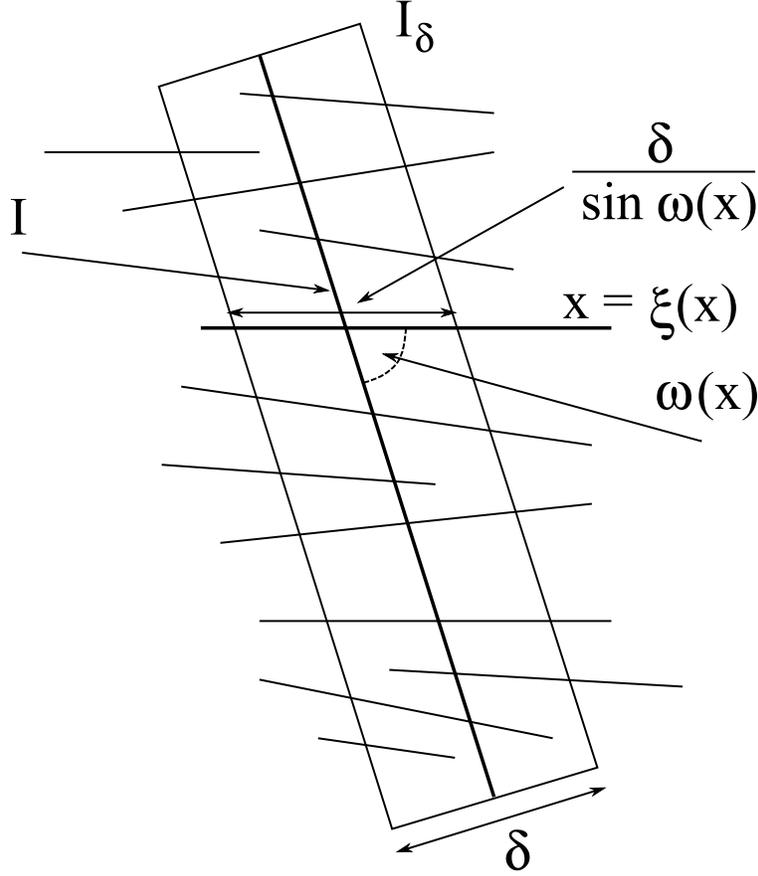}
  \caption{Strip $I_\delta$ for the proof of Lemma \ref{Ry14}. }
  \label{fig:strip}
\end{figure}

\begin{proof} Fix some small $\delta>0$. Let $I_\delta$ be a strip of width $\delta$ ($\delta/2$ on each side of $I$). We note that if $x\in D^s(\delta)\cap I$ 
and $\dis(x,\partial I)>\delta$, then $\xi(x)\cap I_\delta$ is an interval of length $\delta/\sin \omega(x)$, where $\omega(x)=\angle(I,x)$. See Figure \ref{fig:strip}.
So $$\nu_x(I_\delta)=\frac \delta{\sin \omega(x)\cdot D(x)}.$$
Let $E$ be a subinterval of  $I$. If $\bar \nu$ is the Lebesgue measure on $[0,1]^2$, we have 
$$\bar\nu(\pi^{-1}(\pi E)\cap I_\delta)=\int_{\pi(E)}\nu_x(x\cap I_\delta)d\bar m= \delta \int_{\pi(E)}\frac 1{\sin \omega(x) } dm .$$
On the other hand, by Corollary \ref{cor1}, $$\bar\nu(\pi^{-1}(\pi E)\cap I_\delta)=\ell_I(E)\cdot \delta+ o(\delta).$$
This proves (\ref{density on I delta}). Since the angles between directions from $J_-$ and $J_+$ are separated from zero the inequality (\ref{ineq: A_4}) is also proved.
\end{proof}

\begin{remark} Condition (III) holds since $\beta$ is finite.
\end{remark}

Thus, we checked the assumptions of Theorems \ref{TH1} and \ref{TH2}. Hence, we have
\begin{theorem} The results of Theorems  \ref{TH1} and \ref{TH2} apply to $G_\alpha$ maps for $3/4<\alpha<1$.
\end{theorem}

\section{Invariant measures for maps $G$.}\label{Sect:Invariant}

We proved the existence of the invariant measures of the form $\phi\cdot m$ for the factor map $T$. Now, we will construct a $G$-invariant measure
$\mu$ such that the projection $\pi_*(\mu)$ onto $X$ coincides with $\phi\cdot m$.

Let $f:[0,1]^2\to\mathbb R$ be a continuous function. We will define $\mu(f)$.
Let $$f^<(p)=\inf_{\xi(p)} f\ ,\ p\in [0,1]^2, $$ and
$$f^>(p)=\sup_{\xi(p)} f\ ,\ p\in [0,1]^2. $$
Both, $f^<$ and $f^>$ are $\Xi$-measurable ($\Xi$ is the $\sigma$-algebra generated by the partition $\xi=\mathcal P^-$).
We define $$\mu(f)=\lim_{n\to\infty}\tilde \mu((f\circ G^n)^<),$$
where $\tilde\mu=\phi\cdot m$.

\begin{figure}[h] 
  \centering
  \includegraphics[bb=0 -1 277 193,width=3.92in,height=2.74in,keepaspectratio]{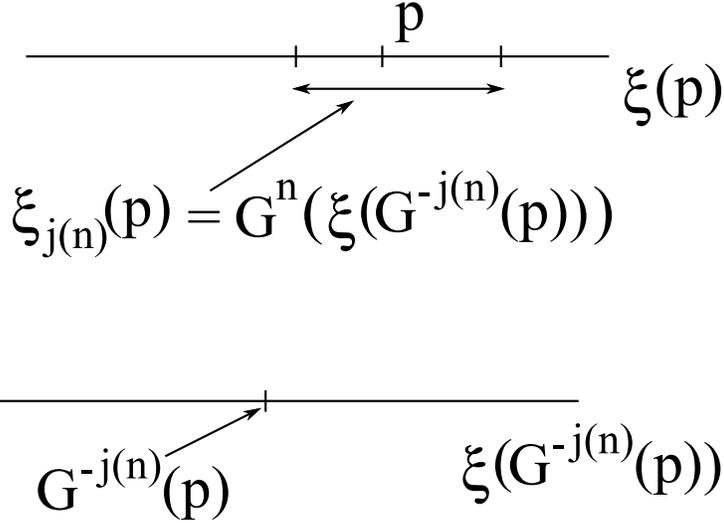}
  \caption{Definition of the function $f_{n|j(n)}^<$.}
  \label{fig:definition-f-less}
\end{figure}

\begin{lemma} \label{Lemma:Ry15}  The limits $\lim_{n\to\infty}\tilde \mu((f\circ G^n)^<)$ and $\lim_{n\to\infty}\tilde \mu((f\circ G^n)^>)$ exist and are equal.
\end{lemma}

\begin{proof}
This proof follows the proof of  Lemma 15 in Rychlik \cite{Ry}, but we have to deal with the fact that $G$ is not invertible. This causes a need for more complicated notation.
The map $G$ has two invertible ``branches" $G_1=G_{|A_1}$ and $G_2=G_{|A_2}$. Corresponding inverses are $G_1^{-1}$ and  $G_2^{-1}$.
Let $j(n)=(i_1,i_2,\dots, i_n)\in \{1,2\}^n$. Then, $G^{j(n)}=G_{i_n}\circ\dots\circ G_{i_2}\circ G_{i_1}$ and $G^{-j(n)}=G_{i_1}^{-1}\circ G_{i_{2}}^{-1}\circ\dots\circ G_{i_n}^{-1}$.
Let us also introduce the notation $(j(n),i_{n+1})=( i_1,i_2,\dots, i_n,i_{n+1})$, for $i_{n+1}\in\{1,2\}$.

We define $f_{n|j(n)}^<(p)=(f\circ G^n)^<\circ G^{-j(n)}(p)$, where $j(n)$ is 
such that $p\in G^{j(n)}([0,1]^2)$. Note that $$f_{n|j(n)}^<(p)= \inf_{q\in \xi(G^{-j(n)}(p))} f(G^n(q))=\inf_{s\in G^n(\xi(G^{-j(n)}(p)))}f(s)$$  is constant on $\xi(G^{-j(n)}(p))$, see Figure \ref{fig:definition-f-less}. Thus, $f_{n|j(n)}^<$ is $\Xi$-measurable.
Also, $f_{n|j(n)}^<\le f_{n+1|(j(n),i_{n+1})}^<$ since $G$ contracts segments $\xi\in \mathcal P^-$.

We define $$f_n^< = \min_{j(n)}f_{n|j(n)}^< .$$
Now, $f_{n}^<$ is $\Xi$ measurable and 
$f_{n}^<\le f_{n+1}^<$.

Similarly, we define  $f_{n|j(n)}^>=(f\circ G^n)^>\circ G^{-j(n)}$ and $$f_n^> = \max_{j(n)}f_{n|j(n)}^> .$$
The functions $f_{n}^>$ are $\Xi$ measurable and 
$f_{n}^>\ge f_{n+1}^>$.

We have $f\ge f_n^<$ for all $n\ge 1$ and $f_1^<\le f_2^<\le \dots\le f_n^<\le \dots$.
Similarly, $f\le f_n^>$ for all $n\ge 1$ and $f_1^>\ge f_2^>\ge \dots\le f_n^>\ge \dots$. Also, if $\xi_n=G^n\xi$ is a partition of $G^n([0,1]^2)$,
then $$ f_n^>(p)-f_n^<(p)= \sup_{\xi_n(p)} f-\inf_{\xi_n(p)} f\le \omega_{\delta_n}(f), $$
where $\delta_n=\sup_p\diam(\xi_n(p))$ and
$$\omega_\delta(f)=\sup_{\dis(x,y)<\delta} |f(x)-f(y)|,$$
 is the modulus of continuity of $f$. Thus, $f_n^>-f_n^<\to 0$ as $n\to\infty$ and, consequently, $f_n^> \searrow f$ and $f_n^<\nearrow f$ uniformly as $n\to\infty$.

We have 
\bes {(f\circ G^n)^<(p)&=\inf_{q\in\xi(p)} f(G^n(q))=\inf_{G^n(\xi(p))} f\\ &\ge \min_{j(n)}\ \ \inf_{G^{j(n)}\left(\xi \left(G^{-j(n)}(G^n(p)))\right)\right)}= (f_n^<\circ G^n) (p),}
so $(f\circ G^n)^<\ge f_n^<\circ G^n$. Similarly, $(f\circ G^n)^>\le f_n^>\circ G^n$.
We have
\bes{&|\tilde\mu ((f\circ G^n)^>)-\tilde\mu ((f\circ G^n)^<)|\le\sup|(f\circ G^n)^>-(f\circ G^n)^<|\\
&\le  \sup|f_n^>\circ G^n-f_n^<\circ G^n|\le \sup|f_n^>-f_n^<|\le \omega_{\delta_n}(f), }
which goes to 0 as $n\to\infty$.
Thus, both limits are the same if they exist. To show existence  we write
$$ f_n^<\circ G^n\le (f\circ G^n)^<\le (f\circ G^n)^>\le f_n^>\circ G^n , $$
which implies 
$$\tilde\mu\left( f_n^<\circ G^n \right)\le \tilde\mu\left((f\circ G^n)^< \right)\le \tilde\mu\left((f\circ G^n)^> \right)\le \tilde\mu\left(f_n^>\circ G^n \right) . $$
By the $T$-invariance of $\tilde\mu$ we have
$$\tilde\mu\left( f_n^<\circ G^n \right)=\tilde\mu\left( f_n^<\circ T^n \right)=\tilde\mu\left( f_n^< \right),$$
and similarly $\tilde\mu\left( f_n^>\circ G^n \right)=\tilde\mu\left( f_n^> \right).$ Since both sequences
$\{f_n^<\}$ and $\{f_n^>\}$ converge uniformly to the same limit we have  $\lim_{n\to\infty} \tilde \mu (f_n^<)=\lim_{n\to\infty} \tilde \mu (f_n^>)$,
which completes the proof.
\end{proof}

\begin{proposition} \label{Prop:Ry11} (Rychlik \cite{Ry}, Proposition 11) Let $\tilde \mu$ be an arbitrary measure on $X$ which is $T$-invariant and such that
the sets of $\Sigma$ are $\tilde \mu$ measurable. Then, there exists  a unique measure on $Y$ such that $\mu$ is $S$-invariant and $\pi_*(\mu)=\tilde\mu$.
\end{proposition}

\begin{proof} Let $\mu$ be constructed  as  in Lemma \ref{Lemma:Ry15}  and let $\eta$ be some other $S$-invariant measure such that $\pi_*(\eta)=\tilde\mu$.
For every continuous function $f$  on $Y$ we have $\eta(f^<)\le\eta(f)\le \eta(f^>)$. Since $\eta(f^<)=(\pi_*\eta)(f^<)=\tilde\mu (f^<)$ (and similarly for $f^>$)
for any function $f$ and in particular for $f\circ S^n$, we get $$\tilde\mu((f\circ S^n)^<)\le\eta(f)\le\tilde\mu((f\circ S^n)^>),$$
as $\eta(f\circ S^n)=\eta(f)$. Going to the limit completes the proof.
\end{proof}

\begin{corollary} In view of Theorem \ref{TH2}, we can construct $G$-invariant measures $\mu_1,\mu_2,\dots,\mu_s$ such that $\pi_*(\mu_i)=\phi_i\cdot m$, $i=1,2,\dots,s$.
\end{corollary}

\begin{theorem}\label{TH4} (Rychlik \cite{Ry}, Theorem 4) Let $\mu$ be a Borel, regular measure on $[0,1]^2$ such that $\pi_*\mu$ is absolutely continuous with respect to $m$. Then,
$$\frac 1n \sum_{k=0}^{n-1} G_*^k\mu \xrightarrow{{\ \ n\to\infty\ \ }}  \sum_{i=1}^s\mu(\bar{C_i})\cdot \mu_i \ ,$$
where $\bar C_i=\pi^{-1}(C_i)$, $C_1,C_2,\dots, C_s$ as in Theorem \ref{TH2}, $\mu_1,\mu_2,\dots,\mu_s$ are as above, and the convergence is in $*$-weak topology of measures.
\end{theorem}


\begin{proof} We refer to  \cite{Ry}.
\end{proof}

\bigskip



\begin{thebibliography}{99}


\bibitem{Alv} Alves, Jos\'e F., Bonatti, Christian, Viana, Marcelo, \textit{ SRB measures for partially hyperbolic
 systems whose central direction is mostly expanding}, Invent. Math. \textbf{140} (2000), no. 2, 351--398. 

\bibitem{Avi} A. Avila, S. Gou\"ezel, and M. Tsujii, \textit{ Smoothness of solenoidal attractors}, Discrete Contin. Dyn. Syst. \textbf{15} (2006), no. 1, 21--35.

\bibitem{Ben} Benedicks, Michael, Young, Lai-Sang, \textit{Sinai-Bowen-Ruelle measures for certain Hénon maps}, Invent. Math. \textbf{112} (1993), no. 3, 541--576.


\bibitem{Bil}  Billingsley, Patrick, \textit{Probability and measure}, Third edition, Wiley Series in Probability and Mathematical Statistics,
 A Wiley-Interscience Publication. John Wiley \& Sons, Inc., New York, 1995.

\bibitem{BG}  Boyarsky A.  and  G\'{o}ra P.  1997 \textit{Laws of Chaos. Invariant
Measures and Dynamical Systems in One Dimension}, Probability and its
Applications, Birkha\"{u}ser, Boston, MA.


\bibitem{Bon}
Bonatti, Christian, Díaz, Lorenzo J., Viana, Marcelo,
\textit{Dynamics beyond uniform hyperbolicity. 
A global geometric and probabilistic perspective}, Encyclopaedia of Mathematical Sciences, \textbf{102}, Mathematical Physics, III. Springer-Verlag, Berlin, 2005, Chapter 11. 


\bibitem{Cow} Cowieson, William, Young, Lai-Sang, \textit{ SRB measures as zero-noise limits}, Ergodic Theory Dynam. Systems \textbf{25} (2005), no. 4, 1115--1138.

\bibitem{MwM1}  Pawe{\l} G\'ora, Abraham Boyarsky, Zhenyang Li and Harald Proppe, \textit{Statistical and Deterministic Dynamics of Maps with Memory},
preprint, http://arxiv.org/abs/1604.06991.





\bibitem{Mah} Maharam, Dorothy,
\textit{On the planar representation of a measurable subfield}, in: Measure theory, Oberwolfach 1983 (Oberwolfach, 1983), 47--57, 
Lecture Notes in Math., \textbf{1089}, Springer, Berlin, 1984. 




\bibitem{Ry} Rychlik, Marek Ryszard, \textit{Invariant Measures and the Variational Principle for Lozi Mappings}
in
\textit{The theory of chaotic attractors. 
Dedicated to James A. Yorke in commemoration of his 60th birthday}, Edited by Brian R. Hunt, Judy A. Kennedy, Tien-Yien Li and Helena E. Nusse. Springer-Verlag, New York, 2004

\bibitem{San} S\'anchez-Salas, Fernando Jos\'e, \textit{Sinai-Ruelle-Bowen measures for piecewise hyperbolic transformations}, Divulg. Mat. \textbf{9} (2001), no. 1, 35--54. 


\bibitem{Sim} Simmons, David, \textit{Conditional measures and conditional expectation; Rohlin's disintegration theorem}, Discrete Contin. Dyn. Syst. \textbf{32} (2012), no. 7, 2565--2582.

\bibitem{Tas} Tasaki, S., Gilbert, Thomas, Dorfman, J. R., \textit{ An analytical construction of the SRB measures for baker-type maps} Chaos and irreversibility (Budapest, 1997), Chaos  \textbf{8} (1998), no. 2, 424--443. 

\bibitem{Tsu1}
Tsujii, Masato,
\textit{Physical measures for partially hyperbolic surface endomorphisms}, 
Acta Math. \textbf{194} (2005), no. 1, 37--132. 

\bibitem{Tsu2} Tsujii, Masato, \textit{ Fat solenoidal attractors}, Nonlinearity \textbf{14} (2001), no. 5, 1011--1027.


\bibitem{You} Young, Lai-Sang, \textit{Bowen-Ruelle measures for certain piecewise hyperbolic maps}, Trans. Amer. Math. Soc. \textbf{287} (1985), no. 1, 41--48. 


\end{thebibliography}
\end{document}